\RequirePackage{fix-cm}
\documentclass[smallextended]{svjour3}       
\smartqed  
\usepackage{mathtools}

\usepackage{tikz-cd}

\usepackage{amsmath,amssymb,amsfonts,latexsym}
\usepackage{mathrsfs}

\usepackage{xypic}
\usepackage{graphicx}
\usepackage{hyperref}

\usepackage{color}
\usepackage{eucal}

\DeclareMathOperator{\Ima}{Im}
\DeclareMathOperator*{\rnk}{rnk}

\begin{document}

\title{On Weighted Simplicial Homology 
}


\author{Thomas  J. X. Li \and
	Christian M. Reidys 
}


\institute{Thomas  J. X. Li  \at
	Biocomplexity Institute, University of Virginia, Charlottesville, VA \\
	\email{jl9gx@virginia.edu}           
	\and 
	Christian M. Reidys \at Biocomplexity Institute and Department of Mathematics, University of Virginia, Charlottesville, VA\\
	\email{cmr3hk@virginia.edu}
}

\date{Received: date / Accepted: date}

\maketitle

\begin{abstract}
	We develop a framework for computing the homology of weighted simplicial complexes with coefficients in a discrete valuation
	ring. A weighted simplicial complex, $(X,v)$,  introduced by Dawson [Cah. Topol. G\'{e}om. Diff\'{e}r. Cat\'{e}g. 31 (1990), pp. 229--243], is a simplicial complex, $X$, together with an integer-valued function, $v$, assigning 
	weights to simplices, such that the weight of any of faces are monotonously increasing. In addition, weighted homology, $H_n^v(X)$, features
	a new boundary operator, $\partial_n^v$.	
	In difference to Dawson, our approach is centered at a natural homomorphism $\theta$ of weighted chain complexes.
	The key object is $H^v_{n}(X/\theta)$, the weighted homology of a quotient of chain complexes induced by $\theta$, appearing
	in a long exact sequence linking weighted  homologies with different weights.
	We shall construct bases for the kernel and image of the weighted boundary map, identifying $n$-simplices as either $\kappa_n$- 
	or $\mu_n$-vertices. Long exact sequences of weighted homology groups and the bases, allow us to prove a structure theorem for the 
	weighted simplicial homology with coefficients in a ring of formal power series $R=\mathbb{F}[[\pi]]$, where $\mathbb{F}$ is a field.  
	Relative to simplicial homology new torsion arises and we shall show that the torsion modules are connected to a pairing between 
	distinguished $\kappa_n$ and $\mu_{n+1}$ simplices.
\end{abstract}

\section{Introduction}


Topology aside, the concept of simplicial complexes is of central importance in a variety of fields including
data analysis and biology. Many real world data-sets exhibit a simplicial structure~\cite{Moore:12,Ramanathan:11,Lin:05} 
and indeed have been organized as such~\cite{Carlsson:09,Spivak:09,Giusti:16}.  While the arising simplicial complexes can 
straightforwardly be studied via topological data analysis (TDA)~\cite{Zomorodian:04,Carlsson:09,Wasserman:18}, a prevalent feature 
of data-sets is the presence of additional simplex-specific data~\cite{Ebli:20}. 

Dawson introduced in 1990~\cite{Dawson:90} the concept of a weighted simplicial complex as a simplicial complex equipped 
with a function $v \colon X \rightarrow R$ such that for simplices $\sigma,\tau 
\in X$ with $\sigma\subseteq \tau $, we have $ v(\sigma)|v(\tau)$. Dawson focused on establishing 
the Eilenberg-Steenrod axioms based on a weighted version of the Mayer-Vietoris sequence and provided a category-theory 
centered treatment. The key difference between standard and weighted simplicial complexes lies in the weighted boundary 
operator that incorporates the weight-function $v$
$$
d^v_n(\sigma)=\sum_{i=0}^n  \frac{v(\sigma)}{v(\hat{\sigma}_i)}\cdot (-1)^i\hat{\sigma}_i,
$$
where $\sigma$ is a $n$-simplex and $\hat{\sigma}_i$ denotes the $i$-th face of $\sigma$. By assumption $v(\hat{\sigma}_i)|v({\sigma})$, 
whence $d^v_n$ is a well-defined boundary map.

Subsequent contributions of Ren \textit{et al}.~\cite{Ren:18} were more application focused, where an extension of Dawsons 
framework to a persistent homology of weighted simplicial complexes was presented, followed by \cite{Wu1:19}, where weighted 
Laplacians were introduced.

Bura \textit{et al}.~\cite{Bura_weighted_21} studied the homology of certain weighted simplicial complexes with coefficients in 
discrete valuation rings, arising from the intersections of loops of a pair of RNA secondary structures~\cite{Bura:21}. \cite{Bura_weighted_21} 
connected weighted simplicial homology with simplicial homology via short exact sequences and a certain chain maps $\theta$. These chain 
maps originated from the inflation map defined in \cite{Bura_weighted_21} that allowed to compute the first weighted homology 
group.

To illustrate how weighted complexes naturally arise and reflecting on~\cite{Dawson:90} and~\cite{Bura_weighted_21}, 
we shall have a closer look at research collaboration networks. These exhibit a simplicial complex structure as follows: 
researchers are considered vertices, and a $n$-simplex in-between $n+1$ researchers appears if those researchers appeared together 
as authors on a paper (by themselves or among others), see Fig.~\ref{F:example}.
 
However, important features cannot be expressed via the simplicial structure alone, as, for instance, the citation number of a simplex. 
i.e.~the number of citations the $n+1$ authors appeared on together. For each simplex, this integer constitutes a 
weight and, by construction, the weight of a face of a simplex is larger than or equal to its weight. The weight 
of  a face, however, does not necessarily divide the weight of its simplex and as a result the weighted homology theory put forward by~\cite{Dawson:90,Ren:18} is not immediately applicable. To incorporate this type of integer-valued weights, arising in a plethora 
of real-world data, we  follow~\cite{Bura_weighted_21} and work with homology with coefficients in discrete valuation rings.


\begin{definition}
	A \emph{weighted simplicial complex}
	is a pair $(X,\omega)$ consisting of a simplicial complex $X$ and a non-negative integer
	function $\omega : X \rightarrow \mathbb{N}$ satisfying 
	\[
	\sigma\subseteq \tau \implies \omega(\sigma)\geq \omega (\tau),
	\]
	for simplices $\sigma,\tau \in X$.
\end{definition}

Given an integral domain $R$ with $\pi\in R\setminus \{0\}$, $(X,\omega)$ naturally induces a weight function $v$ by setting $v(\sigma)=\pi^{\omega(\sigma)}$ and taking the reciprocal of the coefficients in  $d^v_n$ produces the weighted boundary 
map $\partial^v_n:C_{n}(X,R)\to C_{n-1}(X,R)$, 
$$
\partial^v_n(\sigma)=\sum_{i=0}^n  \frac{v(\hat{\sigma}_i)}{v(\sigma)}\cdot (-1)^i\hat{\sigma}_i= \sum_{i=0}^n \pi^{\omega(\hat{\sigma}_i) -\omega(\sigma)} \cdot (-1)^i\hat{\sigma}_i,
$$
where the \emph{weighted chain complex} $C_{n}(X,R)$ is the free $R$-module generated by
all $n$-simplices of $X$.

The \emph{weighted homology}  $H_{n}^v(X)$ of  $(X,\omega)$ is then the $R$-module
$H_{n}^v(X) =\ker \partial_{n}^v / \Ima \partial_{n+1}^v $. Clearly, weighted homology is a generalization of the
standard simplicial homology, since the weighted homology of a complex having constant weighting is isomorphic 
to its simplicial homology.
Furthermore, it is straightforward to see that, over discrete valuation rings, weighted complexes defined via $d^v_n$ and $\partial^v_n$ produce equivalent homology theories.


\begin{figure}[h]\label{F:example}
	\centering
	\includegraphics[width=0.7\textwidth]{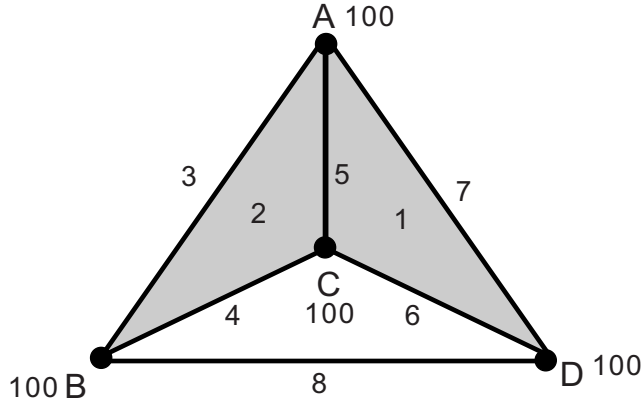}
	\caption
	{\small 
	Weighted simplicial complex, $(X,\omega)$, of a research collaboration network composed by filled (gray) and
	empty (white) triangles.
Suppose $A,B,C,D$ represent four authors that have not appeared as co-authors on any papers, however, $\{A,B,C\}$ or $\{A,C,D\}$ have written papers together. Suppose that $\{A,B,C\}$  has been cited twice, while $\{A,C,D\}$ has been cited once, i.e., $\omega(ABC)=2$ and $\omega(ACD)=1$. Furthermore, any pair appears as authors on some paper, such that the respective citation numbers are given by 
$\omega(AB)=3, \omega(BC)=4,\omega(AC)=5,\omega(CD)=6,\omega(AD)=7,\omega(BD)=8$. Furthermore, suppose each individual author has been cited $100$ times. Then the first simplicial homology group of the complex is given by $H_1(X)\cong \mathbb{Z}$ and
the first weighted homology group is given by $H_{1}^v(X)\cong R\oplus R/(\pi)\oplus R/(\pi^4)$. 
Note that the free submodule of $H_{1}^v(X)$ satisfies $\rnk H_{1}^v(X) =\rnk  H_1(X)$ and the torsion is determined by the differences of 
citation numbers between pairs of simplices $  AB,ABC$ and $AC,ACD$. 
}
\end{figure}

In this paper, we establish an exact sequence relating simplicial and weighted simplicial complexes.
The main result of this paper is a structure theorem for the weighted simplicial  homology with coefficients in
a ring  $R$ of formal power series over a field $\mathbb{F}$, i.e., $R=\mathbb{F}[[\pi]]$. 

To this end, we utilize the chain map $\theta$, which produces homomorphisms between weighted 
homology groups with respect to different weights on the same simplicial complex. The $\theta$-map  generalizes 
the inflation map employed to relate simplicial and weighted homology in case of bi-structures~\cite{Bura_weighted_21}.
In case $R=\mathbb{Q}[[\pi]]$, we prove that $\theta$  is an  injective mapping from simplicial homology with 
integer coefficients  to the weighted homology over $R$ if and only if the integral simplicial homology has no torsion.
The $\theta$-map gives rise to new homology groups, $H^v_{n}(X/\theta)$, constructed via quotients of chain 
complexes. We establish a long exact sequence linking weighted homologies having two different weights, connected via
$H^v_{n}(X/\theta)$. Here $H^v_{n}(X/\theta)$ is a weighted analogue of the relative homology of a pair and our long exact 
sequence is a weighted analogue of  the long exact sequence for a pair. In case of $R=\mathbb{F}[[\pi]]$, we proceed by 
constructing distinguished bases for the kernel $H^v_{n}(X^n)$ and the image 
$\partial_n^v(X)$ of the weighted boundary map. Such bases do not exist in homology with integer coefficients and split
the set of $n$-simplices into $\kappa_n$- and $\mu_n$-simplices. 
We provide an algorithm producing $n$-cycles $\hat\beta_{\kappa}$ each of which containing exactly one distinguished 
$\kappa_n$-simplex such that $\{\hat\beta_{\kappa_n}\mid \kappa_n\}$ forms a basis of $H^v_{n}(X^n)$ and the set of 
$\mu_n$-simplices forms a basis of $\partial_n^v(X)$. We show that the coefficients of $\hat\beta_{\kappa_n}$ can be reduced to $\mathbb{F}$, which in turn, using Nakayama's Lemma, facilitates the efficient computation of weighted 
homology modules \cite{softwareNeelav}.
We are then in position to prove the structure theorem for the weighted
simplicial  homology. Specifically,  we shall prove that the rank of the weighted simplicial  homology equals that of the 
simplicial  homology with coefficients in $R$, and provide a combinatorial interpretation for the torsion of weighted 
homology.  We show that there exists a pairing between $\kappa_n$-  and $\mu_{n+1}$-simplices of dimension $n$ and $(n+1)$, 
such that the torsion modules stem from primary ideals determined by the difference of weights of each respective pair.

We finally present a case study, where we apply the structure theorem to RNA bi-structures~\cite{Bura_weighted_21}. 
This produces a different, short proof for the weighted homology of the loop complex of an RNA bi-secondary structure~\cite{Bura_weighted_21}.

The paper is organized as follows: in Section~\ref{S:property}, we show that $\theta\colon H_n(X)\rightarrow H_n^v(X)$ is injective if and 
only if $H_n(X)$ has no torsion and establish a long exact sequence for weighted homologies having different weights.
 In Section~\ref{S:com}, we construct the $\kappa_n$- and $\mu_n$-basis for the kernel and image of the weighted boundary map, 
 $\partial_n^v$. In Section{~\ref{S:main} we prove the structure theorem for weighted homology and  in Section~\ref{S:case}, we apply our 
 results to RNA bi-structures.

\section{First properties of weighted homology}\label{S:property}

Given 	weighted complexes $(X,v')$ and $(X,v)$, we define $\theta_n^{v',v}\colon C_{n}(X,R)\xrightarrow{} C_{n}(X,R)$ 
by $\theta_n^{v',v}(\sigma)=\frac{v(\sigma)}{v'(\sigma)} \sigma$. By abuse of notation we shall write 
$\theta_n=\theta_n^{v',v}$.

\begin{lemma}\label{L:commute}
	Let $\theta_n:C_{n}(X,R)\xrightarrow{} C_{n}(X,R)$, 
	$\theta_n(\sigma)=\frac{v(\sigma)}{v'(\sigma)} \sigma$, 
	then we have the commutative diagram
	\[
\diagram
\cdots  \rto & C_{n}(X,R)   \rto^{\partial^{v'}_n} \dto^{\theta_{n}} & C_{n-1}(X,R)  
\rto  \dto^{ \theta_{n-1}} &\cdots\\
\cdots 	\rto & C_{n}(X,R)  	\rto^{\partial^v_n} & C_{n-1}(X,R)  	\rto &\cdots
\enddiagram
\]
\end{lemma}
\begin{proof}
	Clearly, 
	$$
	\theta_{n-1}^{v',v}\circ \partial_n^{v'}(\sigma)  = \sum_{i=0}^n (-1)^i 
	\frac{v(\hat\sigma_i)}{v'(\hat\sigma_i)}\cdot \frac{v'(\hat\sigma_i)}{v'(\sigma)} 
	\cdot \hat\sigma_i=\sum_{i=0}^n (-1)^i \frac{v(\sigma)}{v'(\sigma)} 
	\frac{v(\hat\sigma_i)}{v(\sigma)}\hat\sigma_i=
	\partial_n^{v}\circ \theta_{n}^{v',v}(\sigma).
	$$
\end{proof}

Since the $\theta_n$ are chain maps, they induce homomorphisms
\begin{lemma}\label{L:induced}
	The chain maps $\theta_{n}$ induce natural homomorphisms
	$$
	\bar\theta_n\colon H_{n}^{v'}(X) \longrightarrow H^v_{n}(X),\quad 
	\bar\theta_n(\sum_ja_j\sigma_j+\text{\rm Im }\partial^{v'}_{n+1})=
	\theta_n(\sum_ja_j \sigma_j) +\text{\rm Im }\partial^v_{n+1}.
	$$
\end{lemma}

The next proposition is straightforward to verify: 

\begin{proposition}\label{P:homcom}
	Suppose we have two chain complexes such that $i^v \circ \theta_{n,A}= 
	\theta_{n,B}\circ i^{v'}$ and $j^v \circ \theta_{n,B}= \theta_{n,C}\circ j^{v'}$, 
	i.e.~we have the commutative diagram
	$$
	\diagram 0 \rto & C_{n}(A,R) \rto^{i^{v'} } \dto^{\theta_{n,A}}    & 
	C_{n}(B,R)\rto^{j^{v'}} \dto^{\theta_{n,B}} & C_{n}(C,R)\dto^{\theta_{n,C}}\rto  & 0 \\
	0 \rto & C_{n}(A,R) \rto^{i^v}                           & C_{n}(B,R)\rto^{j^v}    & C_{n}(C,R)\rto & 0
	\enddiagram
	$$ 
	Then we have the commutative diagram of long exact homology sequences
	\[
	\diagram 
	\cdots \rto & H_{n}^{v'}(B) \rto \dto^{\bar\theta_{n,B}}  & H^{v'}_{n}(C)  \rto^{\delta_n^{v'}}  \dto^{\bar\theta_{n,C}} & H_{n-1}^{v'}(A)   \rto  \dto^{ \bar\theta_{n-1,A}}  & H_{n-1}^{v'}(B)  \rto \dto^{\bar\theta_{n-1,B}} & \cdots\\
	\cdots  \rto & H^v_{n}(B)  \rto  & H^v_{n}(C)  \rto^{ \delta^v_n} & H^v_{n-1}(A)   \rto &
	H^v_{n-1}(B) \arrow[r] & \cdots
	\enddiagram.
	\]
	and in particular $\bar \theta_{n-1,A}\circ \delta^{v'}_n =\delta^v_n\circ\bar\theta_{n,C}$ as well as $\bar\theta_{n,C}\circ j^{v'} = j^{v}\circ \bar\theta_{n,B}$.
\end{proposition}

Each simplicial complex can be equipped with a constant weight by setting  $v'(\sigma)=1_R$ (the multiplicative identity of $R$) for any $\sigma \in X$. Accordingly, we obtain the chain map $\theta_n :C_{n}(X)\xrightarrow{} C_{n}(X,R)$ given by
$
\theta_n^{v',v} (\sigma) =	v(\sigma) \sigma,
$
and the induced homomorphism   $\bar\theta_{n} :  H_n(X) \xrightarrow{}   H^v_{n}(X) $ between the simplicial homology  and the weighted homology.

\begin{theorem}\label{T:torsion}
	Let $(X,v)$ be a weighted complex with coefficients in $R=\mathbb{Q}[[\pi]]$. 
	Then the following assertions are equivalent:\\
	$(a)$ $\bar\theta_n$ induces the short exact sequence
	$$
	\diagram
	0 \rto & H_n(X) \rto^{\bar\theta_n} &  H^v_{n}(X) 
	\enddiagram
	$$
	$(b)$ $H_n(X)$ has no torsion.
\end{theorem}
\begin{proof}
	$(a) \Rightarrow (b)$: we show that if $H_n(X)$ has torsion, then $\bar\theta_n$ is not injective. 
	Suppose there exists some nontrivial $\sum_ia_i\sigma_i+\text{\rm Im }\partial_{n+1}$ such that 
	$q(\sum_ia_i\sigma_i+\text{\rm Im }\partial_{n+1})=0$. Then
	$q(\sum_ia_i\sigma_i)  =  \partial_{n+1}(\sum_j z_j \tau_j)$ is equivalent to
	$
	\sum_i v(\sigma_i) a_i\sigma_i  = \partial_{n+1}^v(\sum_j \frac{z_j}{q} v(\tau_j) \tau_j).
	$
	Consequently, 
	$$
	\bar\theta_n(\sum_i a_i\sigma_i +\text{\rm Im }\partial_{n+1})=
	\theta_n(\sum_i a_i\sigma_i) + \text{\rm Im }\partial^v_{n+1}=0.
	$$
	
	$(b) \Rightarrow (a)$: let $\sum_i a_i \sigma_i + \text{Im }\partial_{n+1} \in H_n(X)$, where $a_i\in\mathbb{Z}$ and $\sigma_i\in C_n(X)$. 
	
	{\it Claim.} Suppose $\theta_n(\sum_{i\in I} a_i \sigma_i) =\partial^v_{n+1}(z)$, then we have for $q_j\in \mathbb{Q}$:
	$$
	\theta_n(\sum_{i\in I} a_i \sigma_i) =\partial^v_{n+1}(\theta_{n+1}(\sum_{j\in J} q_j\tau_j)).
	$$
	
	To prove the Claim, let $z=\sum_{j\in J}b_j\tau_j$, where $z\in C_{n+1}(X,R)$. 
	We compute
	$$
	\partial_{n+1}^v(z)  =   \sum_{j,k} b_j  (-1)^k \frac{v(\hat\tau_{j,k})}{v(\tau_j)} \hat\tau_{j,k}  =  \sum_i \left[\sum_{\sigma_i\subset \tau_j} c_{i,j} b_j \frac{v(\sigma_i)}{v(\tau_j)}\right] \sigma_i.
	$$
	Then
	\begin{equation}\label{E:x1}
	\sum_h a_h v(\sigma_h) \sigma_h= 
	\sum_h \left[\sum_{\sigma_h\subset \tau_j} c_{h,j} b_j \frac{v(\sigma_h)}{v(\tau_j)}\right] 
	\sigma_h,
	\end{equation}
	where $\{\sigma_h\}$ is the set of faces of the set of simplices $\{\tau_j\}$ and $a_h=0$ for $h\not\in I$.
	We write $v(\tau_j)=\pi^{m_j}$ and $b_j=\sum_n x_{j,n}\pi^n$, where $x_{j,n}\in\mathbb{Q}$ and 
	reformulate eq.~(\ref{E:x1}) via power series
	\begin{eqnarray}\label{E:power}
	\sum_h a_h \sigma_h & = & 
	\sum_h \left[\sum_{\sigma_h\subset \tau_j} \sum_n c_{h,j}x_{j,n}\pi^{n-m_j}\right] \sigma_h .
	\end{eqnarray}
	Eq.~(\ref{E:power}) implies that $r_j =x_{j,m_j}\pi^{m_j}$ has the property
	$$
	\theta_n(\sum_i a_i\sigma_i)=\partial_{n+1}^v(\sum_jb_j\tau_j) = \partial_{n+1}^v(\sum_jr_j\tau_j)
	$$ 
	By construction, any $r_j \equiv 0 \ \text{\rm mod } \pi^{m_j}$ which implies
	$$
	\partial_{n+1}^v(\sum_jr_j\tau_j)= \partial_{n+1}^v(\theta_{n+1}(\sum_{j} x_{j,m_j} \tau_j))
	$$
	and setting $q_j=x_{j,m_j}$ the Claim follows.

	Consequently $\zeta=\sum_{j} q_j\tau_j$ has the property 
	$
	\theta_n(\sum_ia_i\sigma_i)=(\partial_{n+1}^v\circ \theta_{n+1})(\zeta).
	$
	Let $q$ denote the smallest common multiple of the denominators of the $q_j$. Then 
	$q \cdot \zeta$ has integer coefficients and we have
	$$
	(\partial_{n+1}^v\circ \theta_{n+1})(q\cdot \zeta)=\theta_n(q\cdot \sum_ia_i\sigma_i).
	$$
	In view of $\partial_{n+1}^v\circ \theta_{n+1}= \theta_{n}\circ\partial_{n+1}$, we derive
	$$
	\theta_n(q\cdot \sum_ia_i\sigma_i)=
	\partial_{n+1}^v\circ \theta_{n+1}(q\cdot \zeta)=\theta_{n}\circ\partial_{n+1}(q\cdot \zeta).
	$$
	Since $\theta_n\colon C_n(X)\rightarrow C_{n}(X,R)$ is injective on $n$-chains, this implies 
	$
	q\cdot \sum_ia_i\sigma_i=\partial_{n+1}(q\cdot \zeta)
	$,
	i.e.~$q\cdot \sum_ia_i\sigma_i$ is a boundary in $H_n(X)$. 
	
	By construction,
	$
	q\cdot  (\sum_ia_i\sigma_i + \text{\rm Im }\partial_{n+1})= q\cdot \sum_ia_i\sigma_i +\text{\rm Im }\partial_{n+1}
	= 0+ \text{\rm Im }\partial_{n+1}
	$,
	whence $q\cdot  (\sum_ia_i\sigma_i + \text{\rm Im }\partial_{n+1})=0$. Since $H_n(X)$ has no torsion 
	this implies $\sum_ia_i\sigma_i + \text{\rm Im }\partial_{n+1}=0$, i.e.~$\sum_ia_i\sigma_i$ is a boundary and thus trivial in $H_n(X)$ and the proof of the theorem is complete.
\end{proof}

Clearly, $\theta^{v',v}_n(C_n(X,R)) \subset C_{n}(X,R)$ and denoting the quotient module by 
$C_{n}(X/\theta^{v',v})=C_{n}(X,R)/\theta^{v',v}_n(C_n(X,R))$ we have the following 
commutative diagram
\begin{equation}\label{E:jj}
\diagram 0 \rto & C_n(X,R) \rto^{\theta^{v',v}_n}  \dto^{\partial_n}    & C_{n}(X,R)\rto^j \dto^{\partial_n^v} & 
C_{n}(X/\theta^{v',v})\dto^{\partial_{n}^v}\rto  & 0 \\
0 \rto & C_{n-1}(X,R) \rto^{\theta^{v',v}_{n-1}} & C_{n-1}(X,R)\rto^{j}      & C_{n-1}(X/\theta^{v',v})
\rto & 0
\enddiagram
\end{equation} 
We  shall write $\theta_n$ instead of $\theta_n^{v',v}$, 
and $C_{n}(X/\theta)$ instead of $C_{n}(X/\theta^{v',v})$.

Let $H_{n}^{v}(X/\theta)$ denote the homology  with respect to the chain complex $\{C_{n}(X/\theta), \partial_n^{v}\}_n$.

\begin{theorem}\label{T:functorial0}
	(a) Let	$(X,v')$ and $(X,v)$ be weighted complexes with coefficients in an integral domain $R$. 
Then we have the long exact homology sequence
	$$
	\diagram
	\rto & H^v_{n+1}(X/\theta) \rto^{\quad \delta_{n+1}^v}  & H^{v'}_{n}(X) 
	\rto^{\theta_n} & H^v_{n}(X) \rto^j& H^v_{n}(X/\theta) \rto^{\delta_n^v} & 
	H^{v'}_{n-1}(X) \rto &
	\enddiagram
	$$
	(b) Suppose $R=\mathbb{F}[[\pi]]$, where $\mathbb{F}$ is a field
	and $v'(\sigma)=1_R$ for any $\sigma$, i.e.,~$H^{v'}_{n}(X)\cong H_n(X,R)$.
	Then the long sequence splits into the exact sequences
	$$
	\diagram
	0 \rto  & H^{}_{n}(X,R) \rto^{\theta_n} & H^{v}_{n}(X) \rto^j& H^{v}_{n}(X/\theta)
	\rto & 0.
	\enddiagram
	$$
\end{theorem}
\begin{proof}
	We consider the commutative diagram of eq.~(\ref{E:jj}). To define the boundary map $\delta_n^v \colon H^v_{n}(X/\theta) \rightarrow H^{v'}_{n-1}(X)$, let 
	$c\in C_{n}(X/\theta)$ be a cycle. Then $c = j(b)$ for some $b \in C_{n}(X,R)$. Since $j(\partial_n^v(b)) = \partial_n^v(j(b)) = \partial_n^v(c) = 0$, we have $\partial^v_n(b)\in \text{\rm Ker}(j)$.
	Thus $\partial_n^v(b) = \theta_{n-1}(a)$ for some $a \in C_{n-1}(X)$, since 
	$\text{\rm Ker}( j) = \text{\rm Im}(\theta_{n-1})$. Furthermore $\partial^{v'}_{n-1}(a) = 0$ since 
	$$
	\theta_{n-2}(\partial_{n-1}(a)) =\partial^v_{n-1}(\theta_{n-1}(a)) = \partial^v_{n-1}(\partial_n^v(b)) = 0
	$$ 
	and 
	$\theta_{n-2}$ is injective. We define $\delta_v^n \colon H^v_{n}(X/\theta) \rightarrow 
	H^{v'}_{n-1}(X)$
	by sending the homology class of $c$ to the homology class of $a$, $\delta^v_n[c] = [a]$. 
	This is well-defined: the element $a$ is uniquely determined by $\partial_n^v(b)$, since 
	$\theta_n$ is injective. A different choice $b'$ for $b$ produces $j(b') = j(b)$, whence 
	$b' - b\in\text{\rm Ker}( j) = \text{\rm Im}(\theta^{v',v}_n)$. Thus $b'- b = \theta_n(a')$ 
	for some $a'$ and $b' = b + \theta_n(a')$. Replacing $b$ by $b+\theta_n(a')$ means 
	to change $a$ to the homologous element $a+\partial^{v'}_{n}(a')$:
	$$
	\theta_{n-1}(a+\partial^{v'}_n(a')) = \theta_{n-1}(a) + 
	\theta_{n-1}(\partial^{v'}_n(a')) = \partial_n^vb + \partial_n^v(\theta_{n}(a') )= \partial_n^v(b + \theta_n(a')).
	$$
	A different choice of $c$ within its homology class, i.e.~$c + \partial^v_{n+1}(c')$ has the following effect: 
	since $c' = j(b')$ for some $b'$, we then have 
	$$
	c + \partial^v_{n+1}(c') = c + \partial^v_{n+1}(j(b')) = c + j(\partial_{n+1}^v(b')) = j(b + \partial_{n+1}^v(b')), 
	$$
	whence $b$ is replaced by $b + \partial_{n+1}^v(b')$, which leaves $\partial_n^v(b)$ and 
	$\theta_{n-1}(a)$ and therefore also $a$ unchanged.
	
	As for exactness, we observe first $\text{\rm Im}(\delta_{n}^v) \subset  \text{\rm Ker}(\theta_{n-1})$.
	By construction, $\theta_{n-1}(a)=\partial_n^v(b)$ and $\partial_n^v(b)$ is by definition trivial in 
	$H^{v'}_{n-1}(X)$. Secondly, $\text{\rm Ker}(\theta_{n-1}) \subset \text{\rm Im}(\delta_{n}^v)$ holds. 
	Given a cycle $a \in H^{v'}_{n-1}(X)$  such that $\theta_{n-1}(a) = \partial_{n}^v(b)$ for some $b \in 
	H^v_{n}(X)$.
	Consider $j(b)$, we immediately observe that $j(b)$ is a cycle, since $\partial_n^v(j(b)) = j(\partial_n^v(b)) = j(\theta_{n-1}((a)) = 0$, and by construction $\delta_n^v([j(b)])=[a]$, whence 
	$\text{\rm Ker}(\theta_{n-1}) \subset \text{\rm Im}(\delta_{n}^v)$ follows.
	
	Therefore we obtain the long exact sequence
	$$
	\diagram
	\rto & H^v_{n+1}(X/\theta) \rto^{\quad \delta_{n+1}^v}  & H^{v'}_{n}(X) 
	\rto^{\theta_n} & H^v_{n}(X) \rto^j& H^v_{n}(X/\theta) \rto^{\delta_n^v} & H_{n-1}^{v'}(X) \rto &
	\enddiagram
	$$
	
	{\it Claim.} We have the short exact sequence $\diagram 0\rto & H_n(X,R) \rto^{\theta_n}& H^v_{n}(X)
	\enddiagram$.
	
	We first observe that, if $\theta_n(\sum_{i\in I} a_i \sigma_i) =\partial^v_{n+1}(z)$, then
	$$
	\theta_n(\sum_{i\in I} a_i \sigma_i) =\partial^v_{n+1}(\theta_{n+1}(\sum_{j\in J} r_j\tau_j)),
	$$
	where $r_j\in R$. Let $z=\sum_{j\in J}b_j\tau_j$, $z\in C_{n+1}(X,R)$. Then
	\begin{equation}\label{E:y1}
	\sum_h a_h v(\sigma_h) \sigma_h= 
	\sum_h \left[   \sum_{\sigma_h\subset \tau_j} c_{h,j} b_j \frac{v(\sigma_h)}{v(\tau_j)}   \right] 
	\sigma_h, 
	\end{equation}
	where $\{\sigma_h\}$ is the set of faces of the set of simplices $\{\tau_j\}$ and $a_h=0$ for $h\not\in I$.
	We write $v(\tau_j)=\pi^{m_j}$, $a_h=\sum_{n} y_{h,n}\pi^n$ 
	and $b_j=\sum_n x_{j,n}\pi^n$, where $x_{j,n}\in\mathbb{F}$.
	 Rewriting eq.~(\ref{E:y1}) via power series we obtain
	\begin{eqnarray}\label{E:power2}
	\sum_h \left[\sum_{n} y_{h,n}\pi^n\right] \sigma_h & = & 
	\sum_h \left[\sum_{\sigma_h\subset \tau_j} \sum_n c_{h,j}x_{j,n}\pi^{n-m_j}\right] \sigma_h 
	\end{eqnarray}
	and eq.~(\ref{E:power2}) implies that $r_j =\sum_{n\ge m_j} x_{j,n}\pi^{m_j}$ has the property
	$$
	\theta_n(\sum_i a_i\sigma_i)=\partial_{n+1}^v(\sum_jb_j\tau_j) = \partial_{n+1}^v(\sum_jr_j\tau_j). 
	$$ 
	Furthermore by construction, for any $r_j$ holds $r_j \equiv 0 \ \text{\rm mod } \pi^{m_j}$ 
	i.e.,~$r_j=r'_j\pi^{m_j}$, whence
	$$
	\partial_{n+1}^v(\sum_jr_j\tau_j)= \partial_{n+1}^v(\theta_{n+1}(\sum_{j} r_j' \tau_j)).
	$$
	As a result we obtain the equality of $n$-chains with coefficients in $R$:
	$$
	\theta_n(\sum_i a_i\sigma_i)=\partial_{n+1}^v(\sum_jb_j\tau_j) =\partial_{n+1}^v(\theta_{n+1}(\sum_{j} r_j' \tau_j))=
	\theta_n\circ \partial_{n+1}(\sum_{j} r_j' \tau_j).
	$$
	Consequently we derive $\sum_i a_i\sigma_i=\partial_{n+1}(\sum_{j} r_j' \tau_j)$,
	 i.e.,~$\sum_i a_i\sigma_i$ is a boundary in $H_n(X,R)$.
	
	As a result the connecting homomorphisms, $\delta_{n+1}^v$, 
	are trivial, whence the long exact sequence splits into the exact sequences
	$$
	\diagram
	0 \rto  & H_n(X,R) \rto^{\theta_n} & H^v_{n}(X) \rto^j& H^v_{n}(X/\theta) \rto & 0.
	\enddiagram
	$$
\end{proof}
\begin{corollary}\label{C:x}
	We have the exact sequence
	$$
	\diagram
	0 \rto  & H_n(X^n,R) \rto^{\theta_n} & H^v_{n}(X^n) \rto^j& H^v_{n}(X^n/\theta) \rto & 0,
	\enddiagram
	$$
	where $X^n$ denotes the $n$-skeleton of $X$.
\end{corollary}

\section{Some combinatorics}\label{S:com}

\begin{lemma}\label{L:Q}
	Let $(X,v)$ be a weighted complex with coefficients in $R=\mathbb{F}[[\pi]]$. 
	Then we have the short exact sequence of $R$-modules
	\begin{equation}\label{E:unit1}
	\diagram
	0\rto & \pi H_{n}(X,R) \rto & H_{n}(X,R) \rto^{\bar\rho}  & H_{n}(X,\mathbb{F})\rto  & 0 ,
	\enddiagram
	\end{equation}
	where the homomorphism $\bar\rho$ is induced by $\rho$, which maps a formal power series $r\in R$ to its constant term $\bar r$.
\end{lemma}
\begin{proof}
	We first show $\text{\rm Ker}(\bar\rho)\subset \pi H_{n}(X,R)$. Suppose 
	$
	\bar\rho(\sum_j r_j\tau_j+\text{\rm Im }\partial_{n+1})=\sum_j \bar r_j\tau_j + 
	\text{\rm Im }\bar\partial_{n+1}=0
	$.
	Then there exists some $\sum_h \bar a_h\mu_h\in C_{n+1}(X,\mathbb{F})$, producing the equality of $n$-chains
	$
	\sum_j \bar r_j\tau_j -\bar\partial_{n+1}(\sum_h \bar a_h\mu_h)=\sum_j \bar x_j\tau_j=0
	$, where each coefficient, $\bar x_j=0$. Clearly
	$$
	\sum_j r_j\tau_j+\text{\rm Im }\partial_{n+1}=
	\left[\sum_j r_j\tau_j-\partial_{n+1,R}(\sum_h  a_h\mu_h)\right]+\text{\rm Im }\partial_{n+1}
	\in \pi H_{n}(X,R),
	$$
	from which $\text{\rm Ker}(\bar\rho)\subset \pi H_{n}(X,R)$ follows. It remains to observe 
	$\pi H_{n}(X,R)\subset \text{\rm Ker}(\bar\rho)$, which is immediate.
\end{proof}
{\bf Remark.}
While $H_n(X,\mathbb{F})$ is free as an  $\mathbb{F}$-module,
 $H_n(X,\mathbb{F})$ is not a free ${R}$-module.
 In fact, by Lemma~\ref{L:Q}, we can derive that, as  an ${R}$-module, $H_n(X,\mathbb{F})$  is full torsion
   and is composed of 
 $m$ copies of $R/ (\pi)$, where $m=\rnk H_{n}(X,R) $.
 Accordingly, the short exact sequence~(\ref{E:unit1}) is not split exact.
 
\begin{theorem}\label{T:kappa-basis}
	Let $(X,v)$ be a weighted complex with coefficients in $R=\mathbb{F}[[\pi]]$. 
	Then there exists a subset of $n$-simplices, $K\subsetneq \{\sigma\mid \sigma\in C_{n}(X,R)\}$
	and a distinguished $K$-basis of $H^v_{n}(X^n)$,  $\hat{\mathfrak{B}}^v_K=\{\hat\beta_\kappa 
	\mid \kappa \in K\}$, such that the following holds\\
	(i) any $K$-set has the same cardinality and $M=\complement{K}$ is a basis of $\partial_n^v(C_{n}(X,R))$,\\
	(ii) each $\hat\beta_\kappa\in \hat{\mathfrak{B}}^v_K$ contains a unique, distinguished 
	simplex $\kappa\in K$, having coefficient one, 
	$$
	\hat\beta_{\kappa}=\sum r_\ell \mu_\ell + \kappa, \quad\text{\rm where } r_\ell  \text{\rm \  are monomials satisfying }  \deg v(\mu_\ell) =\deg r_\ell v(\kappa), 
	$$
	(iii) let  $\theta_n(\beta_\kappa)=v(\kappa)\hat \beta_\kappa$, 
then	$\mathfrak{B}^v_K=\{\beta_\kappa\mid \kappa\in K\}$ is a basis of $H_n(X^n,R)$,\\
	(vi) let $\gamma_\kappa=\bar\rho(\beta_\kappa)$, then $\{\gamma_\kappa\mid \kappa\in K\}$ 
	is a basis of $H_n(X^n,\mathbb{F})$ and $H_n(X^n,R)$.
\end{theorem}
\begin{proof}
	We construct $\hat{\mathfrak{B}}^v_K$ recursively via the following procedure: set 
	$M_0=\varnothing$ and $S_0=\{\sigma\mid \sigma\in C_{n}(X,R)\}$. Label the 
	simplices $\sigma_i$ arbitrary and examine them one by one, producing recursively 
	the sequence $(M_i,S_i)$, where $M_1=M_0\cup \{\mu_1\mid \mu_1=\sigma_1\}$ and 
	$S_1=S_0\setminus \{\sigma_1\}$, i.e.~we remove $\sigma_1$ from $S_0$, relabel as 
	$\mu_1$ and add to $M_0=\varnothing$.
	
	Having constructed $(M_m,S_m)$ we proceed by examining $\sigma_{m+1}$. 
	We set $S_{m+1}=S_m\setminus \{\sigma_{m+1}\}$ and given the equation
\begin{equation}\label{Eq:kappa}
	\partial_n^v(\sum_\ell r_\ell \mu_\ell + r_{m+1}\sigma_{m+1})=0,
\end{equation}
	distinguish two scenarios.
	In case there exists no nontrivial solution of $r_\ell, r_{m+1}\in R$, we set $M_{m+1}=M_m \cup \{\mu_{n+1}=\sigma_{m+1}\}$. 
	Otherwise, clearing the $\gcd$ of $r_\ell$ and $ r_{m+1}$,  we either have $\sigma_{m+1}$ has coefficient one or some 
	$\mu_{\ell}$ does. In the former case we set $M_{m+1}=M_m$ and in the latter 
	$$
	M_{m+1}=(M_m\setminus \{\mu_\ell \})\cup \{\mu_{m+1}=\sigma_{m+1}\},
	\quad S_{m+1}=S_{m}\setminus \{\sigma_{m+1}\}.
	$$
	Accordingly we either add a new $\mu$-simplex or replace a previously added 
	$\mu$-simplex, while step by step examining all $n$-simplices. In this process
	we have $\partial_n^v(M_{m})\subset \partial_n^v(M_{m+1})$, since a $\mu$-simplex
	replaced in $M_m$ is by construction a linear combination of $M_{m+1}$ $\mu$-simplices.
	
	The procedure terminates in case of $S_t=\varnothing$ and all simplices have been examined.
	$M_t$ is by construction a basis of $\partial_n^v(C_{n}(X,R))$ inducing the bipartition into 
	the set of $\mu$-simplices, $M$, and the complimentary set of $\kappa$-simplices, $K$.
	Since any $\partial_n^v(C_{n}(X,R))$-basis has the same size, all $K$-sets have the
	same cardinality.
	
	For each $\kappa$ there exist unique coefficients $r_\ell\in R$, 
	such that $\hat\beta_{\kappa}=\sum r_\ell \mu_\ell + \kappa$ is a $H_{n}^v(X^n)$-cycle
%
	and the $\hat\beta_{\kappa}$-cycles are linearly independent: 
	$0=\sum_\kappa \lambda_\kappa \hat\beta_\kappa$ implies $\lambda_\kappa=0$ for all 
	$\kappa$, since the simplex $\kappa$ appears uniquely in $\hat\beta_\kappa$.

	{\it Claim $1$.} $\hat{\mathfrak{B}}^v_K=\{\hat\beta_{\kappa}\mid \kappa\in K\}$ 
	is a basis of $H^v_{n}(X^n)$.
	
	Let $c=\sum_h a_h\sigma_h$ be a $H^v_{n}(X^n)$-cycle. By construction, $c$ contains
	at least one $\kappa$-simplex. We prove by induction on the number of distinct $\kappa$-simplices 
	contained in $c$ that $c=\sum_\kappa \lambda_\kappa \hat\beta_\kappa$.
	In case of the induction basis $c$ contains exactly one $\kappa$-simplex, $\kappa_0$. Then
	$c$ contains the summand $r_{\kappa_0}\kappa_0$ and exclusively $\mu$-simplices, otherwise.
	Clearly, $c- r_{\kappa_0}\cdot \hat\beta_{\kappa_0}=c'$ is a cycle containing only $\mu$-simplices 
	which is, by construction, trivial, whence $c= r_{\kappa_0}\cdot \hat\beta_{\kappa_0}$.
	For the induction step assume $c$ contains $(m+1)$ $\kappa$ simplices, $\kappa_1,\dots\kappa_{m+1}$.
	Suppose $c$ has the summand $r_{\kappa_{m+1}}$. Then $c-r_{\kappa_{m+1}}\hat\beta_{\kappa_{m+1}}$
	is a cycle that contains exactly $m$ $\kappa$-simplices since $\hat\beta_{\kappa_{m+1}}$ 
	contains, besides $\kappa_{m+1}$, only $\mu$-simplices. By induction hypothesis we then have
	$c-r_{\kappa_{m+1}}\hat\beta_{\kappa_{m+1}}= \sum_{i=1}^m r_{\kappa_i}\hat\beta_i$ and Claim $1$ follows.
	
	{\it Claim $2$.} 
	For each $\hat\beta_{\kappa}=\sum r_\ell \mu_\ell + \kappa$, there exist  monomials
	 $r_\ell$   satisfying $ \deg v(\mu_\ell) =\deg (r_\ell v(\kappa))$ for any $\ell$. 

As a $H_{n}^v(X^n)$-cycle,  $\hat\beta_{\kappa}$ satisfies $ \partial_n^v(\hat\beta_{\kappa})=\partial_n^v(\sum r_\ell \mu_\ell + \kappa)=0$.
For  any $\hat \beta_\kappa$-face $\sigma$,  we derive 
\begin{align*}
\sum_{\sigma\subset \mu_\ell } c_{\ell} 
\frac{r_\ell}{v( \mu_\ell )} + c_{\kappa} 
\frac{1}{v( \kappa)} =0, \text{ for }   \sigma \subset  \kappa, &\qquad
\sum_{\sigma\subset \mu_\ell } c_{\ell} 
\frac{r_\ell}{v( \mu_\ell )} =0, \text{ for }    \sigma \not \subset  \kappa,
\end{align*}
where $c_{\ell} $ and $ c_{\kappa} $ are $\pm 1$.
We write $v(\mu_\ell )=\pi^{\omega(\mu_\ell)}$,  $v(\kappa)=\pi^{\omega(\kappa)}$ and $r_\ell =\sum_{n} x_{\ell,n}\pi^n$, where $x_{\ell,n}\in\mathbb{F}$.
Rewriting the equations  we obtain
\begin{align*}
\sum_{n}  \sum_{\sigma\subset \mu_\ell } c_{\ell} 
x_{\ell,n}\pi^{n-\omega(\mu_\ell)}+ c_{\kappa}  \pi^{-\omega(\kappa)}
&=0 \qquad \text{for }   \sigma \subset  \kappa \\
\sum_{n}  \sum_{\sigma\subset \mu_\ell } c_{\ell} 
x_{\ell,n}\pi^{n-\omega(\mu_\ell)}& =0\qquad \text{for }    \sigma \not \subset  \kappa.
\end{align*}
In particular,  taking $ [\pi^{-\omega(\kappa)}] $-terms, we derive
\begin{align*}
\sum_{\sigma\subset \mu_\ell } c_{\ell}  x_{\ell,\omega(\mu_\ell)-\omega(\kappa)} + c_{\kappa}  &=0 \qquad \text{for }   \sigma \subset  \kappa \\
\sum_{\sigma\subset \mu_\ell } c_{\ell} 
x_{\ell,\omega(\mu_\ell)-\omega(\kappa)} & =0\qquad \text{for }    \sigma \not \subset  \kappa.
\end{align*}
Let  $m_\ell  = x_{\ell,\omega(\mu_\ell)-\omega(\kappa)} \pi^{\omega(\mu_\ell)-\omega(\kappa)}$ 
be the monomials  obtained by taking $ [\pi^{\omega(\mu_\ell)-\omega(\kappa)}] $-terms of  $r_\ell$. 
Then  $\hat\beta_{\kappa}'=\sum m_\ell \mu_\ell + \kappa$ is by construction a  $H_{n}^v(X^n)$-cycle,
and therefore $\hat\beta_{\kappa}'= \hat\beta_{\kappa} $ since $\hat\beta_{\kappa} $ is unique.
Accordingly,  $r_\ell= m_\ell$, i.e.,    
 $r_\ell$ are monomials   satisfying $ \deg v(\mu_\ell) =\deg (r_\ell v(\kappa))$. 


	
		{\it Claim $3$.} 
	$\mathfrak{B}^v_K=\{\beta_\kappa\mid \kappa\in K\}$ is a basis of $H_n(X^n,R)$, and $\{\gamma_\kappa\mid \kappa\in K\}$ 
		is a basis of $H_n(X^n,\mathbb{F})$ and $H_n(X^n,R)$.
		
By definition,	$\beta_\kappa=\theta_n^{-1}(v(\kappa)\hat \beta_\kappa)
=  \sum \frac{r_\ell v(\kappa)}{v( \mu_\ell ) } \mu_\ell +  \kappa $. Since    $r_\ell$ 
satisfy $ \deg v(\mu_\ell) =\deg (r_\ell v(\kappa))$ by Claim 2, $\beta_\kappa$ is well-defined.
Note that  $\sum_i\lambda_i\beta_i=0$ implies 
	$0=\sum_i\lambda_i \theta_n(\beta_i)=\sum_i\lambda_iv(\kappa)\hat\beta_i$ and hence 
	$\lambda_i v(\kappa)=0$ for all $i$, from which $\lambda_i=0$ follows.
	
	To prove $\{\beta_\kappa\mid \kappa \in K\}$ generates $H_{n}(X^n,R)$, we observe
	that $\kappa$ retains coefficient one in $\beta_\kappa$. In view of this we proceed as in Claim $1$ by induction on the number of distinct $\kappa$-edges contained in a $H_n(X^n,R)$-cycle.
	
	Analogously we can show, using Lemma~\ref{L:Q}, that $\{\bar\rho(\beta_\kappa)\mid \kappa\in K\}$ 
	is a basis of $H_n(X^n,\mathbb{F})$, observing that $\kappa$ appears exclusively in 
	$\bar\rho(\beta_\kappa)$ having coefficient one. Lemma~\ref{L:Q} and 
	Nakayama's Lemma\footnote{Let $M$ be a finitely generated module  over a local ring $R$ with maximal ideal $m$. Then 
		every minimal set of generators of $M$ is obtained from the lifting of some
		  basis of $M/m M$.}  imply that 
	$\{\bar\rho(\beta_\kappa)\mid \kappa\in K\}$ is also a basis of $H_n(X,R)$,  whence 
	Claim $3$.
	
	Therefore $\hat{\mathfrak{B}}^v_K=\{\hat\beta_\kappa \mid \kappa\in K\}$ is a basis of $H^v_{n}(X^n)$
	satisfying $(i)$-$(vi)$ and the proof of Theorem~\ref{T:kappa-basis} is complete.
\end{proof}

{\bf Remark.} (a) The $K$-bases of $H^v_{n}(X^n)$,  $\{\hat\beta_\kappa 
\mid \kappa \in K\}$, depend on $\mathbb{F}$, since $\mathbb{F}$ factors into whether or not 
eq.~(\ref{Eq:kappa}) has a nontrivial solution in $R=\mathbb{F}[[\pi]]$.\\
(b) The above proof can be generalized to the case where $R$ is a ring of polynomials over a field, 
i.e.~$R=\mathbb{F}[\pi]$. \\
(c) 
In case $R$ is a discrete valuation ring, whose uniformizer $\pi$ is \emph{algebraic}, we can 
construct the $K$-basis  $\{\hat\beta_\kappa \mid \kappa \in K\}$ as in Theorem~\ref{T:kappa-basis}, however,
in general the basis does not satisfy properties  $(ii)$-$(vi)$.

\begin{corollary}\label{C:xxx}
	Let $\hat{\mathfrak{B}}^v_K$ be a $K$-basis of $H^v_{n}(X^n)$. Then
	\begin{equation}\label{E:what}
	H^v_{n}(X^n/\theta)\cong \bigoplus_{\kappa\in K}R/(v(\kappa)).
	\end{equation}
\end{corollary}
\begin{proof}
	The projection $p\colon C_{n}(X,R)\rightarrow \oplus_\sigma R/v(\sigma)$, given by
	$\sum_i a_i\sigma \mapsto \sum_i (a_i+v(\sigma)) \sigma$ has kernel $\theta_n(C_n(X,R))$ and consequently $C_{n}(X,R)/\theta_n(C_n(X,R))\cong \oplus_\sigma R/v(\sigma)$. 
Since $v(\kappa)\hat\beta_{\kappa}=\theta_n(\beta_{\kappa})$,
	 each $\hat\beta_{\kappa}$ generates a cyclic $H^v_{n}(X^n/\theta)$ submodule isomorphic to 
	$R/(v(\kappa))$, from which the Corollary follows.
\end{proof}

\section{The main theorem}\label{S:main}

\begin{lemma}\label{L:evolution3}
		Let $(X,v)$ be a weighted complex with coefficients in $R=\mathbb{F}[[\pi]]$. 
	Given $v$, we consider the sequence of weight functions $(v_0,v_1,\dots,v_t=v)$ 
	defined by $v_r(\sigma)=v(\sigma)$ for $\text{\rm dim}(\sigma)\le r$ and $v_r(\sigma)=1$, otherwise.
	Then there exist the exact sequences
	\begin{equation}\label{Eq:shortEX1}
	\diagram
	0 \rto  & H_n(X,R) \rto^{\bar\eta_n^n} & H^{v_n}_{n}(X) \rto^j& \oplus_{\kappa}R/(v(\kappa))
	\rto & 0
	\enddiagram
	\end{equation}
	\begin{equation}\label{Eq:shortEX2}
	\diagram
	0 \rto &  \oplus_{\mu}R/(v(\mu))\rto & H^{v_{n-1}}_{n-1}(X) \rto^{\bar\eta_{n-1}^n} & H^{v_n}_{n-1}(X) \rto & 0,
	\enddiagram
	\end{equation}
	where $\bar\eta_n^r $ is induced from 
	$$
	\eta_n^r (\sigma) = \theta_n^{v_{r-1},v_r} (\sigma) = \frac{v_r(\sigma)}{v_{r-1}(\sigma)}=
	\begin{cases}
	v(\sigma) \sigma & \text{\rm if dim}(\sigma)= r \\
	\sigma & \mbox{\rm otherwise.}
	\end{cases}
	$$ 
\end{lemma}
\begin{proof}
	
	By construction of $v_n$, the quotient $C_{\ell}(X,R)/\eta_{\ell}^n(C_{\ell}(X,R))$ is only nontrivial for $\ell=n$, in which case $C_{n}(X,R)/\eta^n_n(C_{n}(X,R))\cong  \oplus_\sigma R/(v(\sigma))$, where the summation is over the set of all $n$-simplices.
	Consequently, the boundary maps $\bar\partial^{v_n}_n$ and $\bar\partial^{v_n}_{n+1}$ are trivial, whence
	$$
	H_{\ell}^{v_n}(X/\eta^n_{\ell})\cong 
	\begin{cases}
	\oplus_\sigma R/(v(\sigma)) & \mbox{\rm for } \ell=n \\
	0 & \mbox{\rm for }  \ell\neq n.
	\end{cases}
	$$ 
	The long homology sequence of Theorem~\ref{T:functorial0} then becomes the five term exact sequence
	\begin{equation*}
	\diagram
	0 \rto & H^{v_{n-1}}_{n}(X) \rto^{\bar\eta_n^n} & H^{v_n}_{n}(X) \rto^j& H_{n}^{v_n}(X/\eta^n)\rto^{\delta_n^{v_{n}}}
	& H^{v_{n-1}}_{n-1}(X) \rto^{\bar\eta_{n-1}^n} & H^{v_n}_{n-1}(X) \rto & 0
	\enddiagram
	\end{equation*}
	where $H^{v_{n-1}}_{n}(X)=H_{n}(X,R)$, since all $v_{n-1}$-weights of $n$- and $(n+1)$-simplices 
	are one.
	By exactness at $H_{n}^{v_n}(X/\eta_n^n)\cong\oplus_\sigma R/(v(\sigma))$ and $H^{v_{n-1}}_{n-1}(X)$, we have
	$\text{\rm Im }j =\text{\rm Ker }\delta_n^{v_n}$ and  $\text{\rm Im }\delta_n^{v_n}= 	\text{\rm Ker } \bar\eta_{n-1}^n$.  
	Since  
	$\eta_{n-1}^n \mid_{C_{n-1}(X,R)} = \text{\rm id}$
	and $\text{\rm id} \circ \partial_n^{v_{n-1}}=\partial^{v_n}_n\circ \eta_n^n$,
	we have 
	$$
	\text{\rm Im }\partial_{n}^{v_n}/\text{\rm Im }\partial_{n}^{v_{n-1}}= \oplus_\mu 
	\langle \partial_{n}^{v_n}(\mu)\rangle /\langle v(\mu) \partial_{n}^{v_n}(\mu)\rangle\cong 
	\oplus_\mu R/(v(\mu))
	$$
	and the sequence
	$$
	\diagram
	0 \rto & \bigoplus_\mu R/(v(\mu))\rto^{\bar\partial_{n}^{v_n}\qquad }
	& H^{v_n}_{n-1}(X^n)/\text{\rm Im }\partial_{n}^{v_{n-1}} \rto^{\qquad \text{\rm proj}} & 
	H^{v_n}_{n-1}(X)\rto & 0
	\enddiagram
	$$
	is exact. Since $H^{v_{n-1}}_{n-1}(X^{n-1})= H^{v_n}_{n-1}(X^{n-1})$ we have 
	$H^{v_n}_{n-1}(X^n)/\text{\rm Im }\partial_{n}^{v_{n-1}}\cong H_{n-1}^{v_{n-1}}(X)$ which provides an interpretation of $\text{\rm Ker }\bar\eta_{n-1}^n$, via 
	$$
	\diagram
	0 \rto &  \bigoplus_\mu R/(v(\mu)) \rto^{\bar\partial_{n}^{v_n}} & H^{v_{n-1}}_{n-1}(X) \rto^{\bar\eta_{n-1}^n=\text{\rm proj}} & H^{v_n}_{n-1}(X) \rto & 0.
	\enddiagram
	$$
	$v$ bipartitions the set of $n$-simplices into $\kappa$- and $\mu$-simplices and
	using the exactness at $H_{n}^{v_n}(X/\eta_n^n)\cong \bigoplus_\sigma R/(v(\sigma))$, we obtain
	$$
	\diagram
	0 \rto & H_{n}(X,R) \rto^{\bar\eta_n^n} & H^{v_n}_{n}(X) \rto^j &  \bigoplus_\kappa R/(v(\kappa)) \rto & 0.
	\enddiagram
	$$
\end{proof}

\begin{theorem}\label{T:structure}
		Let $(X,v)$ be a weighted complex with coefficients in $R=\mathbb{F}[[\pi]]$. 
Let $F_{n}^v$ and $T_{n}^v$ denote the free and the torsion submodules of 
		$H^v_{n}(X)$. 
Then	there exists an exact sequence
	$$
	\diagram
	0\rto & H_{n}(X,R) \rto^{\bar\theta_n} & F_n^v \rto^{j\qquad} &  \bigoplus_{s=1}^q R/(v(\kappa_s^n))\rto & 0,
	\enddiagram
	$$
	where $\{\kappa^n_1,\dots,\kappa_q^n\}\dot\cup \{\kappa^n_{q+1},\dots,\kappa_{p}^n\}=K$ is a 
	distinguished bipartition of the $\kappa$-simplices of dimension $n$. Furthermore,
	$\text{\rm rnk}_{R}(H^v_{n}(X)) =\text{\rm rnk}_{\mathbb{F}}(H_{n}(X,\mathbb{F})) $ and 
	\begin{equation}\label{E:func}
	T_n^v\cong \bigoplus_{s=q+1}^{p} R/(v(\kappa^n_s)/v(\mu^{n+1}_{\alpha(s)})),
	\end{equation}
	where $\alpha \in S_{p-q}$ establishes a pairing between $\kappa_s^n$- and 
	$\mu_{\alpha(s)}^{n+1}$-simplices of dimension $n$ and $(n+1)$, respectively.
\end{theorem}
\begin{proof}
	By the general structure theorem of finitely generated modules over pids, we have $H^v_{n}(X)\cong 
	F^v_{n}\oplus T^v_{n}$. Furthermore, we have $H^v_{n}(X^n)\cong \mathfrak{F}^v_{n}\oplus \mathfrak{T}^v_{n}$,
	where $\mathfrak{F}^v_{n}\cong_\phi {F}^v_{n}$, $\phi(f)=f+\text{\rm Im }\partial_{n+1}^v$ and 
	$\mathfrak{T}^v_{n}/\text{\rm Im }\partial_{n+1}^v\cong T^v_n$. This follows from the diagram
$$
\diagram
0\rto &  \text{\rm Ker }(p) \rto & H^v_{n}(X^n) \drto^p\rto^{p_1} & H^v_{n}(X)\dto^{p_2} \rto &  0\\
          &                                     &                                      & F_n^v \rto  & 0.
\enddiagram        
$$	
Here $\mathfrak{T}^v_{n} =\text{\rm Ker }(p)$ and since $F_n^v$ is free, it is projective and we have 
$H^v_{n}(X^n)=\mathfrak{T}^v_{n}\oplus \mathfrak{F}^v_{n}$ with $\mathfrak{F}^v_{n}/\text{\rm Im }\partial^v_{n+1}\cong {F}^v_{n}$. Finally, by construction, we have $\text{\rm rnk}(\text{\rm Im }\partial_{n+1}^v)=\text{\rm rnk}(\mathfrak{T}_n^v)$.
	
	Let $\varphi^n=\eta^{n+1}\circ \eta^n$, we note that $\varphi^n_n=\theta_n$ since both maps
	coincide on $n$- and $(n+1)$-simplices.
	
	{\it Claim $1$.} We have the exact sequence
	$$
	\diagram
	0\rto & H_{n}(X,R) \rto^{\bar\theta_n} &H^{v_{}}_{n}(X) \rto^{j\qquad\qquad \qquad\qquad} &  
	(\bigoplus_{\kappa^n} R/(v(\kappa^n))/\bar\partial_{n+1}^{v_{n+1}} (\bigoplus_{\mu^{n+1}} R/(v(\mu_{n+1})))\rto &0  .
	\enddiagram
	$$
	By Theorem~\ref{T:functorial0} we have the long exact sequence of homology groups
	\begin{equation}\label{E:jj!}
	\diagram
	H^{v_{n+1}}_{n+1}(X/\varphi^n) \rto^{\delta_{n+1}^{v_{n+1}}} & H^{v_{n-1}}_{n}(X) \rto^{\bar\varphi_n^n} & H^{v_{n+1}}_{n}(X) \rto^j& H_{n}^{v_{n+1}}(X/\varphi^n)\dto^{\delta_n^{v_{n+1}}} \\
	& H_{n-1}^{v_{n+1}}(X/\varphi^n) & \lto^j H^{v_{n+1}}_{n-1}(X)  &  H^{v_{n-1}}_{n-1}(X) \lto^{\bar\varphi_{n-1}^n} 
	\enddiagram
	\end{equation}
	By construction $\bar\varphi_n^n=\bar\theta_n$, $H^{v_{n-1}}_{n}(X) \cong H_n(X,R)$, $H^{v_{n+1}}_{n}(X)=H^{v_{}}_{n}(X)$ and
	$H_{n-1}^{v_{n+1}}(X/\varphi^n) =0$.
	In view of 
	$
	C_{n+1}(X/\varphi^n)\cong \bigoplus_{\tau^{n+1}} R/(v(\tau^{n+1}))
	$, where the direct sum ranges over all $(n+1)$-simplices, $\tau^{n+1}$, we have
	$$
	\bar\partial_{n+1}^{v_{n+1}}(C_{n+1}(X/\varphi^n))=\bar\partial_{n+1}^{v_{n+1}}(\bigoplus_{\mu^{n+1}}R/(v(\mu^{n+1})) 
	\cong \bigoplus_{\mu^{n+1}}R/(v(\mu^{n+1})),
	$$
	where the summation ranges over all $\mu^{n+1}$-simplices which form a basis of $\partial_{n+1}^{v_{n+1}}(X)$.
	Since $C_{n-1}(X/\varphi^n)=0$ we obtain $\bar\partial_n^{v_{n+1}}\colon C_{n}(X/\varphi^n) \rightarrow 0$, where $C_{n}(X/\varphi^n)\cong \bigoplus_{\sigma^n} R/(v(\sigma^n))$.
	Using $\bar\partial_n^{v_{n+1}}\circ\bar\partial_{n+1}^{v_{n+1}} =0$, we derive
	\begin{equation}\label{E:33}
	H_{n}^{v_{n+1}}(X/\varphi^n) \cong \left(\bigoplus_{\kappa^n} R/(v(\kappa_n))/\bar\partial_{n+1}^{v_{n+1}} (\bigoplus_{\mu^{n+1}}R/(v(\mu^{n+1})))\right)\oplus \left(\bigoplus_{\mu^{n}}R/(v(\mu^{n}))\right).
	\end{equation}
	By Lemma~\ref{L:evolution3} we have the exact sequence
	$$
	\diagram
	0 \rto & \bigoplus_{\mu^{n}} R/(v(\mu^{n})) \rto^{\quad\bar\partial_n^{v_{n+1}}} & H^{v_{n-1}}_{n-1}(X) \rto^{\bar\varphi_n^n} & 
	H^{v_{n+1}}_{n-1}(X)\rto & 0,
	\enddiagram
	$$
	which combined with exactness of eq.~(\ref{E:jj!}) at $H_{n}^{v_{n+1}}(X/\varphi^n)$ and eq.~(\ref{E:33}) gives rise to the exact sequence of Claim $1$:
	$$
	\diagram
	0\rto & H_{n}(X,R) \rto^{\bar\theta_n} &H^{v}_{n}(X) \rto^{j\qquad\qquad\qquad} &  
	(\bigoplus_{\kappa^n} R/(v(\kappa^n))/\bar\partial_{n+1}^{v_{n+1}} (\bigoplus_{\mu^{n+1}} R/(v(\mu^{n+1})))\rto & 0  
	\enddiagram
	$$
	and Claim $1$ follows.
	
	We proceed by dissecting the exact sequence of Claim $1$ into the free and torsion modules.
	
	{\it Claim $2$.} We have the exact sequence
	$$
	\diagram
	0\rto & H_{n}(X,R) \rto^{\bar\theta_n} & F_n^v(X) \rto^{j\quad} &  
	\bigoplus_{s=1}^q R/(v(\kappa^n_s)) \rto & 0 ,
	\enddiagram
	$$
	and  $\text{\rm rnk}_{R}(H^v_{n}(X)) =\text{\rm rnk}_{\mathbb{F}}(H_{n}(X,\mathbb{F})) $.

	In view of $\theta_n(H_n(X,R))\subset F_n^v(X)$ and Theorem~\ref{T:functorial0}, we have 
	$$
	\diagram
	0\rto & H_n(X,R) \rto^{\bar\theta_n} & F_n^v(X).
	\enddiagram
	$$
	Furthermore, by Theorem~\ref{T:functorial0} and Corollary~\ref{C:xxx}, 
	$$
	\diagram
	0\rto & H_n(X^n,R) \rto^{\theta_n} & H^v_{n}(X^n) \rto^{j\quad} &
	\bigoplus_{\kappa^n} R/(v(\kappa^n))\rto & 0.
	\enddiagram
	$$
	By restriction, $j$ induces the surjective homomorphism
	$j\mid_{\mathfrak{F}_{n}}\colon \mathfrak{F}^v_{n} \rightarrow \bigoplus_{s=1}^q R/(v(\kappa^n_s))$ and
	$$
	\diagram
	0\rto & H_n(X,R) \rto^{\theta_n} &  F_n^v(X) \rto \dto^{\phi^{-1}} & 
	\bigoplus_{s=1}^q R/(v(\kappa^n_s))\rto & 0.\\
	&                                       & \mathfrak{F}^v_{n} \urto_{j\mid_{\mathfrak{F}^v_{n}}\quad} & &
	\enddiagram
	$$
	Since $	\bigoplus_{s=1}^q R/(v(\kappa^n_s))$ is full torsion, the exact sequence implies $\text{\rm rnk}_{R}(H^v_{n}(X)) =\text{\rm rnk}_{{R}}(H_{n}(X,R) )$. Combing with $\text{\rm rnk}_{\mathbb{F}}(H_{n}(X,\mathbb{F}))  =\text{\rm rnk}_{{R}}(H_{n}(X,R) )$ derived by  Lemma~\ref{L:Q}, we have $\text{\rm rnk}_{R}(H^v_{n}(X)) =\text{\rm rnk}_{\mathbb{F}}(H_{n}(X,\mathbb{F})) $,
	whence Claim $2$.

	{\it Claim $3$.} We have 
	$$
	T_n^v(X)\cong \bigoplus_{s=q+1}^{|K|} R/(v(\kappa^n_s)/(v(\mu^{n+1}_{\alpha(s)})).
	$$
	We consider the homomorphism embedding $\text{\rm Im }\partial^v_{n+1}$ into
	$\mathfrak{T}_n^v$.
	Since $R$ is pid, there exists a $\mathfrak{T}_n^v$-basis, $\mathfrak{B}_1=
	\{\hat{\mathfrak{t}}_{q+1},\dots \hat{\mathfrak{t}}_p\}$ and a $\text{\rm Im }\partial^v_{n+1}$-basis 
	$\mathfrak{B}_0=\{x_s\cdot \hat{\mathfrak{t}}_s\mid s=q+1,\dots,p\}$, where $x_s\in R$ represent the invariant factors.
	
	Claim $3$ follows from two observations that put these bases into context with
	Corollary~\ref{C:x} and Corollary~\ref{C:xxx}. First, since $\varphi^n_n$ elevates $n$- as well as
	$(n+1)$-simplices to their $v$-weight, we have
	$$
	\text{\rm Im }\partial^v_{n+1}/\theta_{n}(\text{\rm Im }\partial_{n+1})=\bar\partial_{n+1}^{v_{n+1}}(C_{n+1}(X/\varphi^n))\cong \bigoplus_{s=q+1}^{p} R/(v(\mu_s^{n+1})).
	$$
	Secondly, using $H^{v}_{n}(X^n)\cong \mathfrak{T}_n^v\oplus \mathfrak{F}_n^v$ and
	the commutative diagram
	$$
	\diagram
	0 \rto & H_{n}(X^n,R)\dto^{\text{\rm id}} \rto^{\theta_n} & \mathfrak{F}_n^v \dto^{\text{\rm inj}}\rto^{j\qquad} &  \bigoplus_{s=1}^q R/(v(\kappa_s^n)) \rto\dto^{\text{\rm inj}} & 0 \\
	0 \rto & H_{n}(X^n,R) \rto^{\theta_n} & H^{v}_{n}(X^n) \rto^{j\qquad} &  \bigoplus_{\kappa^n} R/(v(\kappa^n)) \rto & 0 \\
	0 \rto & \text{\rm Im }\partial_{n+1}\uto^{\text{\rm inj}} \rto^{\text{\rm res }\theta_{n}} &\uto^{\text{\rm inj}}  \mathfrak{T}_{n}^v \rto^{\text{\rm res }j\qquad} &  \bigoplus_{s=q+1}^{p}R/(v(\kappa^n_s)) \uto^{\text{\rm inj}}\rto & 0
	\enddiagram
	$$
	we arrive at
	$$
	\mathfrak{T}_{n}^v(X)/\theta_n(\text{\rm Im }\partial_{n+1})\cong \bigoplus_{s=q+1}^{p}
	R/(v(\kappa_s^n)).
	$$
	In order to see how the $\kappa_s^n$ and $\mu_s^{n+1}$ align, we consider the commutative diagram
	$$
	\diagram
	0 \rto & \text{\rm Im }\partial_{n+1}\dto^{\text{\rm id}} \rto^{\text{\rm res }\theta_{n}} &\dto^{\psi}  \mathfrak{T}_{n}^v \rto^{\text{\rm res }j\qquad} &  \bigoplus_{s=q+1}^{p}R/(v(\kappa^n_s)) \dto^{\bar\psi}\rto & 0\\
	0 \rto & \text{\rm Im }\partial_{n+1} \rto^{\text{\rm res }\theta_{n}} &  \text{\rm Im }\partial^v_{n+1} \rto^{\text{\rm res }j\qquad} &  \bigoplus_{s=q+1}^{p}R/(v(\mu^{n+1}_s)) \rto & 0
	\enddiagram
	$$
	where we extend $\psi(\mathfrak{t}_s)=x_s\cdot \mathfrak{t}_s$ linearly to an $R$-module homomorphism $\psi$.
	Choosing the $\mathfrak{T}_n^v$- and $\text{\rm Im }\partial^v_{n+1}$-bases $\mathfrak{B}_1$ and
	$\mathfrak{B}_0$, respectively, we have
	$
	\mathfrak{T}_{n}^v(X)/\theta_n(\text{\rm Im }\partial_{n+1}) \cong\sum_s 
	\langle \mathfrak{t}_s+\theta_n(\text{\rm Im }\partial_{n+1} )\rangle
	$
	as well as 
	$
	\text{\rm Im }\partial^v_{n+1}/\theta_n(\text{\rm Im }\partial_{n+1}) \cong\sum_s 
	\langle x_s\mathfrak{t}_s+\theta_n(\text{\rm Im }\partial_{n+1} )\rangle
	$. 
	Since $R$ is a discrete valuation ring,  $\langle \mathfrak{t}_s+\theta_n(\text{\rm Im }\partial_{n+1} )\rangle$ and $\langle x_s\mathfrak{t}_s+\theta_n(\text{\rm Im }\partial_{n+1} )\rangle$ are primary modules and as such indecomposable, whence
	for each $q+1\leq s\leq p$
	$$
	\langle \mathfrak{t}_s+\theta_n(\text{\rm Im }\partial_{n+1} )\rangle\cong R/(v(\kappa^n_{s_1})) \quad \text{\rm and }\quad \langle x_{s}\mathfrak{t}_{s}+\theta_n(\text{\rm Im }\partial_{n+1} )\rangle\cong R/(v(\mu^{n+1}_{s_2})). 
	$$
	By the commutativity of the right square,  
	$$
	\bar\psi( R/(v(\kappa^n_{s_1}))) \cong \bar\psi(\langle \mathfrak{t}_s+\theta_n(\text{\rm Im }\partial_{n+1} ))=\langle x_{s}\mathfrak{t}_{s}+\theta_n(\text{\rm Im }\partial_{n+1} )\rangle \cong R/(v(\mu^{n+1}_{s_2})).
	$$
	Thus  there exists some permutation $\alpha$  that pairs $\kappa^n_{s_1}$ with $\mu^{n+1}_{s_2}$ with $\alpha(s_1)=s_2$ such that
	\[
	\bar\psi( R/(v(\kappa^n_{s_1}))) \cong R/(v(\mu^{n+1}_{s_2})),
	\]
	and as a result we arrive at
	\begin{align*}
	\mathfrak{T}_{n}^v(X)/\text{\rm Im }\partial^v_{n+1} &\cong 
	\Big[ \mathfrak{T}_{n}^v(X)/\theta_n(\text{\rm Im }\partial_{n+1})\Big] \Big/ \Big[  \text{\rm Im }\partial^v_{n+1}/\theta_{n}(\text{\rm Im }\partial_{n+1})\Big] \\
	& \cong \bigoplus_{s=q+1}^{p} \Big[ R/(v(\kappa^n_s)) \Big] \Big/  \Big[  R/ (v(\mu^{n+1}_{\alpha(s)})) \Big]\\
	& \cong \bigoplus_{s=q+1}^{p} R/(v(\kappa^n_s)/v(\mu^{n+1}_{\alpha(s)})).
	\end{align*}
\end{proof}

{\bf Remark.}
In view of the structure theorem, let us revisit the weighted simplicial complex $(X,v)$ depicted in 
Figure~\ref{F:example}. Based on Theorem~\ref{T:kappa-basis}, we compute the  $K$-basis of 
$H^v_{1}(X^1)$ given by $\hat{\mathfrak{B}}^v_K=\{\hat\beta_{AC},\hat\beta_{CB},\hat\beta_{BA}\}$ 
with $K=\{AC,CB,BA \}$, where
\begin{align*}
\hat\beta_{AC} &= AC+\pi CD+\pi^2 DA  \\
\hat\beta_{CB} &= CB +\pi^4 BD+\pi^2 DC\\
\hat\beta_{BA}&=BA+\pi^4 AD+\pi^5 DB.
\end{align*}
The $\mu^2$-simplices are given by $ABC,ACD$ and thus $\partial_{2}^v (X)=\{\partial_{2}^v(ABC),\partial_{2}^v(ACD) \} $. 
By Theorem~\ref{T:structure}, we derive a partition $K=\{BA,AC\}\dot\cup \{CB\}$ 
and a pairing $\alpha\colon  \{BA,AC\} \xrightarrow{}  \{ABC,ACD\} $ with $\alpha(BA)=ABC, \alpha(AC)=ACD$.
Then the torsion of the first weighted homology $H^{v}_{1}(X)$ is given by
\[	
T_1^v \cong   R/(\pi^{\omega(BA)-\omega(ABC)})\bigoplus R/(\pi^{\omega(AC)-\omega(ACD)})\cong  R/(\pi)\oplus R/(\pi^4).
\]
Since $\rnk H^{v}_{1}(X)= \rnk H_{1}(X,R)=1$, we obtain $H^{v}_{1}(X)\cong R \oplus R/(\pi)\oplus R/(\pi^4)$.

\section{Case study: RNA bi-structures}\label{S:case}



RNA is a biomolecule that folds into a helical configuration of its sequence by forming base pairs. The most prominent class of coarse-grained structures are the RNA secondary	structures~\cite{Waterman:78s,Waterman:78aa}.
A secondary structure can be uniquely decomposed into loops and the free energy of a structure is calculated as the sum of the energy of its individual loops~\cite{Zuker:81}.
	
A  \emph{bi-structure} $(S,T)$ is a pair of secondary structures $S$ and $T$ over the same backbone. We represent a bi-structure as a diagram on a horizontal backbone with the $S$-arcs drawn in the upper and the $T$-arcs drawn in the lower half plane. 
Two arcs $(i, j)$ and $(k, l)$ are crossing if  $i<k<j<l$. 
Crossing induces an equivalence relation for which nontrivial equivalence classes are called \emph{crossing components}.
 A vertex $k$ is \emph{covered} by an arc  $(i, j)$ if $i\leq k\leq j$ and there exists no other arc $(p,q)$ such that $i<p<k<q<j$. A \emph{loop} is  the set of vertices covered by an arc $(i, j)$, in particular, the exterior loop is given by the set of vertices  covered by an artificial rainbow arc connecting the first and last vertices.
The \emph{loop complex}, $K(S,T)$,  is
 the nerve formed by $S$-loops and $T$-loops of a bi-structure $(S,T)$. 
 The loop complex $X=K(S,T)$ can be augmented by assigning a weight to each simplex of $X$, 
 where the weight  encodes the cardinality of intersections of loops in the simplex, see Fig.~\ref{F:bi}.

\begin{figure}[h]
	\centering
	\includegraphics[width=0.9\textwidth]{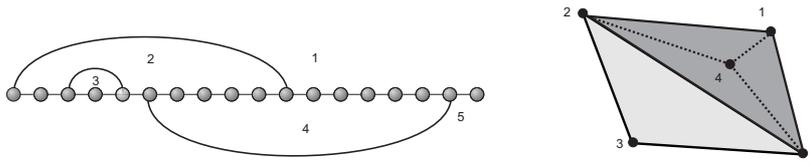}
	\caption
	{\small 
		LHS:  A  bi-structure $(S,T)$ with $S$-loops $1,2,3$ and $T$-loops $4,5$, where $1$ and $5$ are exterior loops,
		and arcs $(1, 11)$ and $(6,17)$ form a crossing component.
		RHS: its corresponding loop complex given by $K(S,T)=\{1,2,3,4,5,[1,2],[1,4],[1,5],[2,3],[2,4],[2,5], [3,5],[4,5],[1,2,4],[1,2,5],[1,4,5],[2,4,5],[2,3,5]\}$.
		The weights assigned to simplices in the loop complex encode the size of  intersections of loops in the simplex.
		While $\omega(1)=9, \omega(2)=10,\omega(3)=3,\omega(4)=12,\omega(5)=8$,
	the weights of $1$-simplices are given by 	$\omega([1,2])=\omega([2,3])=\omega([4,5])=2,\omega([1,4])=7,\omega([1,5])=3,\omega([2,4])=6,\omega([2,5])=6,\omega([3,5])=3$
		and  the weights of $2$-simplices are
		$\omega([1,2,4])= \omega([1,2,5])=\omega([1,4,5])=\omega([2,4,5])=1,\omega([2,3,5])=2$.
	}
	\label{F:bi}
\end{figure}

 \cite{Bura_weighted_21}  computed the weighted homology  for the loop complex of RNA bi-structures.
 In particular, \cite{Bura_weighted_21}  showed that the weighted simplicial complex of an arbitrary bi-structure
can be transformed via Whitehead moves~\cite{Whitehead:39} to a complex, which does not contain any $3$-simplices or $2$-simplices 
having weight greater than $1$. Referring to such complexes as lean, the following holds:

\begin{theorem}\label{T:bistr1}\cite{Bura_weighted_21}
	Let $(X,v)$ be a lean, weighted loop complex of a bi-structure $(S,T)$, where $v(\sigma)=\pi^{\omega(\sigma)}$ is given by the size of the intersection of loops. 
	Let $\hat{\mathfrak{B}}^{v}_K$ be $K$-basis of $H^{v}_{1,R}(X^1)$ and 
	$M=\complement{K}=\{\mu_s\}$ be a basis of $\partial_1^v(X)$.
	Then 
	\begin{align*}	
	H_{2}^v(X) &\cong R^C\\
	H_{1}^v(X) &\cong \oplus_{\kappa\in K} R/ (\pi^{\omega(\kappa)-1})\\
	H_{0}^v(X) &\cong R \oplus \bigoplus_{\mu_{\alpha(s)} \in M} R/ \Big(\pi^{\omega(v_{s})-\omega(\mu_{\alpha(s)} )}\Big),
	\end{align*}
	where  $C$ denotes the number of  crossing components in $(S,T)$, $v_s$ is a $0$-simplex of $X$ and
	the pairing $(v_s, \mu_{\alpha(s)})$ between $0$-simplices $v_s$
	and $1$-simplices  $\mu_{\alpha(s)} \in M$ is given by Theorem~\ref{T:structure}.
\end{theorem}
This result can  be derived from our structure theorem as follows:
\begin{proof}
	For simplicial homology with integer coefficients, \cite{Bura_weighted_21} proved that the loop complex of a bi-structure satisfy $H_2(X)=\mathbb{Z}^C$, 
	$H_1(X)=0$ and    $H_0(X)=\mathbb{Z}$. Combing with 	$\text{\rm rnk}_{R}(H^v_{n}(X)) =\text{\rm rnk}_{\mathbb{F}}(H_{n}(X,\mathbb{F})) $ by Theorem~\ref{T:structure}, we have $\rnk H_{2}^v(X)  =C$, 
	$\rnk 	H_{1}^v(X)  =0$ and $\rnk H_{0}^v(X)  =1$.

	Since the lean complex $X$ contains no $3$-simplices, $H_{2}^v(X) $ is free, whence $H_{2}^v(X) \cong R^{C}$.
	
	Let $\{\delta \mid \delta \in \Delta \}$ denote the set of $2$-simplices in $X$. Since $\rnk H_{1}^v(X)=0$, 
	Theorem~\ref{T:structure} shows there exists
	a bijection $p$  between the set  $K$ of $1$-simplices $\kappa$ and the set of $2$-simplices,
	i.e., the pairings $( \kappa_{i},\delta_{p(i)})$ for each $ \kappa_{i}\in K$.
	Since each $2$-simplex in a lean complex has weight $1$, i.e. $v(\delta)=\pi$,
	we have $\frac{v(\kappa_{i})}{v(\delta_{p(i)})}= \pi^{\omega(\kappa_{i})-1}$.
	Theorem~\ref{T:structure} establishes 
	that $H_{1}^v(X)\cong \bigoplus_{\kappa\in K} R/ (\pi^{\omega(\kappa)-1})$.

	Similarly, Theorem~\ref{T:structure} provides the pairing $\alpha$ between $0$-simplices $v_s$
	and $1$-simplices  $\mu_{\alpha(s)} \in M$, i.e., $(v_s, \mu_{\alpha(s)})$.
	Consequently, the torsion of $H_{0}^v(X)$ is given by $T_0^v\cong \bigoplus_{\mu_{\alpha(s)} \in M} R/ 
	\Big(\pi^{\omega(v_{s})-\omega(\mu_{\alpha(s)} )}\Big)$, completing the proof.
\end{proof}

{\bf Remark.} We can extend the above analysis to $\tau$-structures~\cite{Li:21}, which can be viewed as RNA-RNA interaction structures and generalize bi-structures.

\section{Declarations}

Funding - The authors received no financial support for the research, authorship, and/or publication of this article.

Conflicts of interest/Competing interests - None.

Availability of data and material - Non Applicable.

Code availability - Non Applicable.

\begin{acknowledgements}
	We gratefully acknowledge the comments and discussions from Andrei Bura, Qijun He and Fenix Huang.
\end{acknowledgements}

\bibliographystyle{spmpsci} 
\bibliography{ref.bib}

\end{document}